\documentclass[12pt,reqno]{amsart}
\usepackage{amssymb}
\usepackage{appendix}
\usepackage{tikz}
\usepackage{cancel}  
\usepackage[normalem]{ulem} 
\usepackage{yhmath}   
\usepackage{framed}  
\usepackage{enumerate} 
\usepackage[margin=2.5cm]{geometry}  
\numberwithin{equation}{section}  
\numberwithin{figure}{section}
\usepackage{hyperref}
\hypersetup{colorlinks=true,
linkcolor=blue,
citecolor=blue,
urlcolor=blue,
}

\newcommand{\rr}{\mathbb{R}}

\newcommand{\zz}{\mathbb{Z}}

\newcommand{\ci}{\mathbb{T}}

\newcommand{\wh}{\widehat}
\newcommand{\p}{\partial}
\newcommand{\ee}{\varepsilon}

\newtheorem{theorem}{Theorem}
\newtheorem{lemma}{Lemma}
\begin{document}
\title{H\"older Continuity of the Data to Solution Map for HR in the
Weak Topology}
\author{\textit{David Karapetyan} \\ University of Notre Dame}
\address{Department of Mathematics  \\
University of Notre Dame\\
Notre Dame, IN 46556 }
\date{}
\begin{abstract}
It is shown that the data to solution map for the hyperelastic rod equation is
H\"older continuous from bounded sets of Sobolev spaces with exponent $s > 3/2$
measured in a weaker Sobolev norm with index $r < s$ in both the
periodic and non-periodic cases. The proof is based on energy estimates coupled
with a delicate commutator estimate and multiplier estimate. 
\end{abstract}
\maketitle
\markboth{H\"older Continuity of the Data to Solution Map for HR in the
Weak Topology}{David Karapetyan}
\section{Introduction}
The hyperelastic rod (HR) equation
\begin{gather}
\p_t u =  -\gamma u \p_x u -
\p_{x} (1 - \p_{x}^{2})^{-1} \left[ \frac{3-\gamma}{2}u^2 +
\frac{\gamma}{2} \left( \p_x u \right)^2
\right], \ \gamma \in \rr \slash \{0\}
\label{hyperelastic-rod-equation}
\\
u(x,0) = u_0(x), \; \; x \in \rr, \; \; t \in \rr
\label{init-cond}
\end{gather}
was first
derived by Dai in \cite{Dai_1998_Model-equations} as a one-dimensional 
model for the propagation of finite-length and
small-amplitude axial deformation waves in thin cylindrical
rods composed of a compressible Mooney-Rivlin
material. Unlike competing models for small-amplitude wave propagation such as
the Korteweg-de Vries equation and its regularized counterpart the Boussinesq equation, for $\gamma \neq 0$ not all traveling wave solutions of HR are smooth. More precisely, there are a variety of traveling wave solutions to the HR equation that
can be obtained using various combinations of peaks, cusps, compactons,
fractal-like waves, and plateaus (see Lenells 
\cite{Lenells_2006_Traveling-waves}). 
\\
\\
For the case $\gamma = 0$, we obtain the 
BBM equation, first proposed by 
Benjamin, Bona, and Mahony 
\cite{Benjamin_1972_Model-equations} as a model for 
the unidirectional evolution of long waves.
Solitary-wave solutions to this 
equation are smooth, global, and orbitally stable (see Benjamin 
\cite{Benjamin_1972_The-stability-o}, Benjamin, Bona, and Mahony
\cite{Benjamin_1972_Model-equations}, and Lenells
\cite{Lenells_2006_Traveling-waves}). For the case $\gamma =1$, we obtain the Camassa-Holm equation, a bi-Hamiltonian
water wave equation first written explicitly by Camassa and Holm
\cite{Camassa-Holm-1993-An-integrable-shallow-water}. There is an extensive
literature devoted to the Camassa-Holm equation. We refer the reader to the work
of Himonas, Kenig, and Misiolek \cite{Himonas:2010p1187}, Himonas and Kenig
\cite{Himonas:2009fk}, Constantin and Lannes
\cite{Constantin-Lannes-2009-The-hydrodynamical-relevance-of-the-Camassa-Holm},
Bressan and Constantin \cite{Bressan_2007_Global-conserva}, 
Molinet \cite{Molinet_2004_On-well-posedne}, Constantin and Strauss
\cite{Constantin_2000_Stability-of-pe}, and the references therein.
\\
\\
The HR equation is well-posed in the sense of Hadamard (existence, uniqueness,
and continuous dependence of solutions on initial data) in Sobolev spaces
$H^{s}$, $s > 3/2$ on both the line and circle, see Yin
\cite{Yin_2003_On-the-Cauchy-p}, Zhou \cite{Zhou_2005_Local-well-pose}, and
Karapetyan \cite{Karapetyan:2010fk}. A natural question is whether or not the
continuous dependence can be strengthened to uniform dependence or better. For
Camassa-Holm, it was shown in \cite{Himonas:2009fk} and \cite{Himonas:2010p1187} 
that dependence is not uniform. The method relied upon the use of energy
estimates to show that that the difference between actual solutions and appropriately chosen approximate solutions with coinciding initial data is small (see
also Koch and Tzvetkov \cite{Koch_2005_Nonlinear-wave-}). Mirroring this method,
non-uniform dependence for HR was shown in \cite{Karapetyan:2010fk}. 
\\
\\
For the Burgers equation, it is also known that for $s > 3/2$, dependence is not
better than continuous. Furthermore, Kato \cite{Kato_1975_Quasi-linear-eq}
showed that for $s > 3/2$ the data to solution map $u_{0} \mapsto u(t)$ is not
H\"older continuous from a closed ball in $H^{s}$ centered at $0$ and measured
in the $H^{r}$ norm, $r < s$, to $C([0, T], H^{r})$, where $T$ depends upon
the $H^{s}$ radius of the ball. More precisely, for fixed $0 < \gamma < 1$
and fixed constant $c > 0$,
there exist solutions $u, v$ of Burgers with bounded initial data in $H^{s}$
(and hence, a common lifespan $T$) and $0 < t_{0} < T$ such that
\begin{equation*}
\begin{split}
\| u(t_{0}) - v(t_{0}) \|_{H^{r}} 
& > c \| u_{0} - v_{0} \|^{\gamma}_{H^{r}}.
\end{split}
\end{equation*}
However, for certain general quasi-linear hyperbolic systems, Kato also obtained
uniform continuity of the data to solution map for initial data in Sobolev
spaces with integer index, measured in a weaker Sobolev norm. More recently, Tao
\cite{Tao:2004p1537} obtained Lipschitz continuity of the data to solution map
for the Benjamin-Ono equation for $H^{1}$ initial data measured in $L^{2}$.
Herr, Ionescu, Kenig, and Koch \cite{Herr:2010p886} have also obtained Lipschitz
continuity in a weaker topology for the Benjamin-Ono with generalized
dispersion. Hence, it is reasonable to ask whether a result similar to these
holds for HR\@. Our main motivation stems from the work of Chen, Liu, and Zhang
\cite{Chen:2011fk} on the
b-family
\begin{gather}
\p_t u =  -u \p_x u -
\p_{x} (1 - \p_{x}^{2})^{-1} \left[ \frac{b}{2}u^2 +
\frac{3-b}{2} \left( \p_x u \right)^2
\right],
\label{b-family}
\\
u(x,0) = u_0(x), \; \; x \in \rr, \; \; t \in \rr
\label{init-cond-b-fam}
\end{gather}
for which they proved H\"older continuity of the data to solution map from a closed ball $B(0, h)$ in $H^{s}$, $s >
3/2$ (measured in the $H^{r}$ topology, $r <s$) to $C([0, T], H^{r})$, with $T
= T(h)> 0$ and H\"older index $\alpha = \alpha(b, s, r)$ given by 
\begin{equation*}
\begin{split}
\alpha = 
\begin{cases}
1, \quad & (s, r) \in \Omega_{1}
\\
1, \quad & b=3 \ \ \text{and} \ \ (s, r) \in \Omega_{2}
\\
2(s-1)/(s-r), \quad  & b\neq 3 \ \ \text{and} \ \ (s, r) \in \Omega_{2}
\\
s-r, \quad & (s, r) \in \Omega_{3}
\end{cases}
\end{split}
\end{equation*}
where
\begin{equation*}
\begin{split}
\Omega_{1} & = \left\{ (s, r): \   s > 3/2, \ 0 \le r \le s-1, \  r \ge 2-s \right\}
\\
\Omega_{2} & = \left\{(s, r): \  3/2 < s < 2, \ 0 < r < 2-s   \right\}
\\
\Omega_{3} & = \left\{ (s,r): \  s > 3/2, \ s-1 < r < s  \right\}.
\end{split}
\end{equation*}
Given this result, and the similarities between the $b$-family and HR (both can
be thought of as weakly dispersive nonlocal perturbations of Burgers), in this
work we study the continuity properties of the
data-to-solution map for the HR equation, expanding upon the work in
\cite{Karapetyan:2010fk}. More precisely, following \cite{Chen:2011fk} we show
the following result:
\begin{theorem}
For $\gamma \neq 0$, the
data to solution map for HR is H\"older continuous from $B_{H^{s}}(R)$ (in
the topology of $H^{r}$) to $C([0, T], H^{r})$, where $T = T(R)$, for $s >
3/2$, $-1 \le r < s$. More
precisely, consider the following sets 
\begin{equation*}
\begin{split}
& \Omega_{1} = \left\{ (s, \ r) \in \rr^{2}:
\ s>3/2, \ -1 \le r \le s-1, \  r \ge 2 -s  \right\}
\\
& \Omega_{2} = \left\{ (s, \ r) \in \rr^{2}:
\ 3/2 < s < 3, \ -1 \le r < 2-s \right\}
\\
& \Omega_{3} = \left\{ (s, \ r) \in \rr^{2}:
\  s>3/2, \  s-1 < r < s  \right\}.
\end{split}
\end{equation*}
\label{thm:main-thm}
\end{theorem}
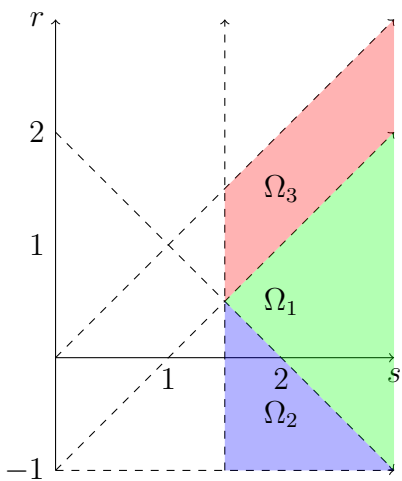
\begin{center}
\begin{tikzpicture}[scale=1.5]
\draw [->] (0,0) -- (3,0) node [below] {$s$};
\draw [->] (0,-1) -- (0,3) node [left] {$r$};
\draw [->, dashed] (0,0) -- (3,3);
\draw [->, dashed] (0,-1) -- (3,2);
\draw [->, dashed] (0,2) -- (3,-1);
\draw [->, dashed] (0,-1) -- (3,-1);
\draw [->, dashed] (3/2,-1) -- (3/2, 3);
\fill[color=green, fill opacity=0.3] (1.5, 0.5) -- (3,2) -- (3,0) -- (3,-1);
\fill[color=red, fill opacity=0.3] (1.5, 0.5) -- (1.5,1.5) -- (3,3) -- (3,2);
\fill[color=blue, fill opacity=0.3] (1.5, 0.5) -- (1.5, -1) -- (3, -1);
\foreach \x/\xtext in {1, 2}
\draw[shift={(\x,0)}]  node[below] {$\xtext$};
\foreach \y/\ytext in {-1, 1, 2}
\draw[shift={(0,\y)}]  node[left] {$\ytext$};
\draw (2,1.5) node {$\Omega_{3}$};
\draw (2,0.5) node {$\Omega_{1}$};
\draw (2,-0.5) node {$\Omega_{2}$};
\end{tikzpicture}
\end{center}
Then for two initial data $u_{0}, v_{0} \in B_{H^{s}}(R)$, there exist unique
corresponding solutions \\ $u(x,t), v(x,t)$ for $0 \le t \le T= T(R)$ to the
HR equation \eqref{hyperelastic-rod-equation} which satisfy 
\begin{equation*}
\begin{split}
\| u(t) - v(t) \|_{H^{r}} \le C \| u_{0} - v_{0} \|_{H^{r}}^{\alpha(s, r)},
\quad 0
\le t \le T
\end{split}
\end{equation*}
where 
\begin{equation*}
\begin{split}
\alpha = 
\begin{cases}
1, \quad & (s,r) \in \Omega_{1} 
\\
2(s-1)/(s-r),  \quad & (s, r) \in \Omega_{2}
\\
s-r, \quad & (s, r) \in \Omega_{3}.
\end{cases}
\end{split}
\end{equation*}
%
%
%
%
%
%
%
%
We remark that this result is sharper then the analogue obtained in
\cite{Chen:2011fk} for the $b$-family. We are confident that the techniques
applied in this paper can be applied to sharpen the results obtained in
\cite{Chen:2011fk}.
\section{Proof of H\"older Continuity}
We note that the only significant difference between the proof of H\"older
continuity in the periodic and non-periodic cases is in the proof of Lemma~\ref{lem:frac-deriv}, which we address in Section~\ref{sec:pf-lemmas}. Hence, we focus our
attention on the proof of H\"older continuity in the periodic case. 
\subsection{Region $\Omega_{1}$} 
\label{ssec:reg-m-imp}
Let $u_{0}(x), v_{0}(x)
\in B_{H^{s}}(R)$, $s > 3/2$ be two initial data. Then from
the well-posedness theory for HR \cite{Karapetyan:2010fk}, we
know that there exists unique corresponding solutions $u, v \in C(I,
B_{H^{s}}(2R))$ to \eqref{hyperelastic-rod-equation}.
Set $v=u-w$. Then $v$ solves the Cauchy-problem
\begin{align}
\label{uniqueness-exp}
& \p_t v
=  -\frac{\gamma}{2} \p_x [v(u + w)] 
\\
\notag
& \phantom{\p_t v = }-\p_x (1 - \p_{x}^{2})^{-1} \left\{
\frac{3-\gamma}{2}[v(u+w)] + \frac{\gamma}{2}[\p_x v \cdot \p_x (u+w)]
\right\},
\\
& v(x,0) = u_{0}(x) - v_{0}(x).
\label{uniqueness-init-data}
\end{align}
Let
\begin{equation*}
D^{m} \doteq (1 - \p_x^2)^{m/2}, \quad m \in \rr.
\end{equation*}
Applying $D^r$ to both sides of \eqref{uniqueness-exp}, then 
multiplying both sides by $D^r v$ and integrating, we obtain
\begin{equation}
\begin{split}
\frac{1}{2} \frac{d}{dt} \|v\|_{H^r}^2
= & -\frac{\gamma}{2} \int_{\ci} D^r \p_x [v(u+w)] \cdot
D^r v \ dx
\\
& - \frac{3-\gamma}{2} \int_{\ci}  D^{r -2}
\p_x[v(u+w)] \cdot
D^r v \ dx  
\\
& - \frac{\gamma}{2} \int_{\ci} D^{r 
-2} \p_x [ \p_x v
\cdot \p_x (u+w)]\cdot D^r v \ dx.
\label{2v}
\end{split}
\end{equation}
We now estimate \eqref{2v} in parts.
\subsubsection{Estimate of Integral 1} 
Note that
\begin{equation}
\begin{split}
& \left |  -\frac{\gamma}{2} \int_{\ci} D^r \p_x [v(u+w)] \cdot
D^r v \ dx \right |
\\
& =
\left |
-\frac{\gamma}{2} \int_{\ci} \left[ D^r \p_x, \ u+w \right]v \cdot
D^r v \ dx - \frac{\gamma}{2} \int_{\ci} (u+w) D^r
\p_x v \cdot D^r v\ dx
\right | \\
& \lesssim \left |
\int_{\ci} \left[ D^r \p_x, \ u+w \right]v \cdot
D^r v \ dx \right |
+ \left | \int_{\ci} (u+w) D^r \p_x v
\cdot D^r v\
dx \right |.
\label{4v}
\end{split}
\end{equation}
Observe that integrating by parts gives
\begin{equation}
\begin{split}
\left | \int_{\ci} (u+w) D^r \p_x v \cdot
D^r v \ dx \right |
\le \|\p_x (u+w)\|_{L^\infty}
\|v\|_{H^r}^2.
\label{4'v}
\end{split}
\end{equation}
To estimate the remaining integral of \eqref{4v}, we shall need the following
following result taken from \cite{Himonas:2010p1187}:
\begin{lemma}
\label{cor1}
If $s > 3/2$ and $-1 \le r  \le s -1$, then
\begin{equation*}
\begin{split}
\|[D^r \p_x ,f]g\|_{L^2} \lesssim \|f\|_{H^s} \|g\|_{H^r}.
\end{split}
\end{equation*}
\end{lemma}
An application of 
Cauchy-Schwartz and Lemma~\ref{cor1} then yields 
\begin{equation}
\begin{split}
\left | \int_{\ci} [D^r \p_x, \ u+w] v
\cdot D^r v \ dx \right |
& \lesssim \|u+w\|_{H^s} 
\|v\|_{H^r}^2.
\label{7v}
\end{split}
\end{equation}
Combining \eqref{4'v} and \eqref{7v} and applying the Sobolev embedding 
theorem, we obtain the estimate
\begin{equation}
\begin{split}
\left |  -\frac{\gamma}{2} \int_{\ci} D^r \p_x [v(u+w)] \cdot
D^r v \ dx \right |
\lesssim \|u+w\|_{H^s} \|v\|_{H^r}^2, \quad s > 3/2, \ -1 \le r \le s-1.
\label{8v}
\end{split}
\end{equation}
\subsubsection{Estimate of Integral 2} We shall need the following.
%
%
%
%
%
%
%
%
\begin{lemma}
For $s > 3/2$, $r \le s$, $s + r \ge 2$, we have
\begin{equation*}
\begin{split}
\| fg \|_{H^{r-1}} \lesssim \| f \|_{H^{r-1}} \| g \|_{H^{s-1}}.
\end{split}
\end{equation*}
\label{lem:frac-deriv}
\end{lemma}
Applying Cauchy-Schwartz and Lemma~\ref{lem:frac-deriv}, we obtain
\begin{equation*}
\begin{split}
\left | - \frac{3-\gamma}{2} \int_{\ci}  D^{r -2}
\p_x[v(u+w)] \cdot
D^r v \ dx  \right |
& \lesssim \|u+w\|_{H^{r -1}} \|v\|_{H^r}^2
\end{split}
\end{equation*}
which implies
\begin{equation}
\begin{split}
\left | - \frac{3-\gamma}{2} \int_{\ci}  D^{r -2}
\p_x[v(u+w)] \cdot
D^r v \ dx  \right |
& \lesssim \|u+w\|_{H^{s}} \|v\|_{H^r}^2
\label{3v}
\end{split}
\end{equation}
for $s > 3/2, \ r \le s, \ \text{and} \ s + r \ge 2$.
\subsubsection{Estimate of Integral 3} We first apply
Cauchy-Schwartz to obtain
\begin{equation*}
\begin{split}
\left | - \frac{\gamma}{2} \int_{\ci} D^{r 
-2} \p_x [ \p_x v
\cdot \p_x (u+w)]\cdot D^r v \ dx \right | 
\lesssim 
\|[\p_x v \cdot \p_x (u+w)] \|_{H^{r -1}}
\|v\|_{H^r}.
\end{split}
\end{equation*}
Applying Lemma~\ref{lem:frac-deriv} and the inequality $\| f_{x}
\|_{H^{m-1}} \le \| f \|_{H^{m}}$,  we conclude that
\begin{equation}
\begin{split}
\left | - \frac{\gamma}{2} \int_{\ci} D^{r 
-2} \p_x [ \p_x v
\cdot \p_x (u+w)]\cdot D^r v \ dx \right | 
\lesssim \|u+w \|_{H^{s}}
\|v\|_{H^r}^2
\label{3'v}
\end{split}
\end{equation}
for $s > 3/2, \ r \le s, \ \text{and} \ s + r \ge 2$.
Grouping \eqref{8v}, \eqref{3v}, and \eqref{3'v}, we obtain
\begin{equation*}
\begin{split}
\frac{1}{2} \frac{d}{dt}
\|v\|_{H^r}^2
& \lesssim \|u+w\|_{H^s}
\|v\|_{H^r}^2, \quad | t | < T
\\
& \le 4R \| v \|_{H^{r}}^{2}.
\label{9v}
\end{split}
\end{equation*}
Letting $y(t) = \| v \|^{2}_{H^{r}}$, we obtain
\begin{equation*}
\begin{split}
\frac{dy}{dt} \le cy
\end{split}
\end{equation*}
where $c = c(s, r, R) > 0$. Hence
\begin{equation*}
\begin{split}
y(t) \le y(0) e^{ct}, \quad | t | < T
\end{split}
\end{equation*}
which implies
\begin{equation*}
\begin{split}
y(t) \le y(0) e^{cT}.
\end{split}
\end{equation*}
Substituting back in for $y$, we see that
\begin{equation*}
\begin{split}
\| v \|_{H^{r}}^{2} \le \| v(0) \|^{2}_{H^{r}} e^{cT}
\end{split}
\end{equation*}
or
\begin{equation}
\label{lip-ineq}
\begin{split}
& \| u(t) - w(t) \|_{H^{r}} \le C \| u_{0} - w_{0} \|_{H^{r}}, 
\\
& \text{for} \ | t | < T,
\ s > 3/2, \ -1 \le r \le s-1, \ s + r \ge 2.
\end{split}
\end{equation}
Hence, in region $\Omega_{1}$, the data to solution map is locally Lipschitz from
$B_{H^{s}}(R)$ (measured with the $H^{r}$
norm) to $C([-T, T], H^{r})$, with Lipschitz constant $C = C(s, r, R)$.
\subsection{Region $\Omega_{2}$} 
\label{ssec:case-4}
We have the estimate
\begin{equation}
\label{fgh}
\begin{split}
\| u(t) - w(t) \|_{H^{r}}
& < \|u(t) - w(t) \|_{H^{2-s}}.
\end{split}
\end{equation}
We see that \eqref{lip-ineq} is valid for $r = 2-s$, $3/2 < s \le 3$.
Hence, applying \eqref{lip-ineq} to \eqref{fgh}, we obtain 
\begin{equation*}
\begin{split}
\| u(t) - w(t) \|_{H^{r}}
\lesssim \|u_{0} - w_{0} \|_{H^{2-s}}.
\end{split}
\end{equation*}
We need the following interpolation
result. 
%
%
%
%
%
%
%
%
\begin{lemma}
For $m_{1} < m < m_{2}$,
\begin{equation*}
\begin{split}
\| f \|_{H^{m}} \le \| f \|_{H^{m_{1}}}^{(m_{2}-m)/(m_{2} - m_{1})} \| f
\|_{H^{m_{2}}}^{(m -m_{1})/(m_{2} - m_{1})}.
\end{split}
\end{equation*}
\label{lem:interp}
\end{lemma}
Applying the lemma with $m_{1} =r$, $m = 2-s$, and $m_{2} = s$ (notice
$m_{2} > m$ for $s > 1$), we bound 
\begin{equation*}
\begin{split}
\| u_{0} - w_{0} \|_{H^{2-s}} 
& \le \| u_{0} - w_{0} \|_{H^{r}}^{\frac{2(s-1)}{s-r}} \| u_{0} - w_{0}
\|_{H^{s}}^{\frac{2-s-r}{s-r}}
\\
&  \lesssim \| u_{0} - w_{0} \|_{H^{r}}^{\frac{2(s-1)}{s-r}}.
\end{split}
\end{equation*}
We conclude that
\begin{equation*}
\begin{split}
\| u(t) - w(t) \|_{H^{r}} \lesssim \|u_{0} - w_{0} \|_{H^{r}}^{\frac{2(s-1)}{s-r}}.
\end{split}
\end{equation*}
\subsection{Region $\Omega_{3}$} 
\label{ssec:case-2}
Applying Lemma~\ref{lem:interp} with $m_{1} = s-1$, $m =r$ and $m_{2} = s$, and 
using the estimate
\begin{equation*}
\begin{split}
\|u - w \|_{H^{s}} \le 4R
\end{split}
\end{equation*}
we obtain
\begin{equation}
\label{pre-lip-ap}
\begin{split}
\| u(t) - w(t) \|_{H^{r}} & \lesssim \| u(t) - w(t) \|_{H^{s-1}}^{s-r} \|u(t)
- w(t)\|_{H^{s}}^{1-s+r}
\\
& \simeq \| u(t) - w(t) \|_{H^{s-1}}^{s-r}.
\end{split}
\end{equation}
We see that \eqref{lip-ineq} is valid for  $r = s-1$, $s \ge 3/2$. Hence,
applying \eqref{lip-ineq} to \eqref{pre-lip-ap} gives
\begin{equation*}
\begin{split}
\| u(t) - w(t) \|_{H^{r}} & \lesssim  \|u_{0} - w_{0}\|_{H^{s-1}}^{s-r} 
\\
& \le
\|u_{0} - w_{0}\|_{H^{r}}^{s-r}.
\end{split}
\end{equation*}
This completes the proof of Theorem~\ref{thm:main-thm}. \qed
%
%
%
%
%
%
%
%
%
%
\section{Proofs of Lemmas} 
\label{sec:pf-lemmas}
\begin{proof}[Proof of Lemma~\ref{lem:frac-deriv}]
For the periodic case we have
\begin{equation*}
\begin{split}
\| fg\|_{H^{r-1}}^{2}
& \le  \sum_{n \in \zz}  (1 + n^{2})^{r-1}\left [ \sum_{k \in \zz}
| \wh{f}(k) |  | \wh{g}(n - k) | (1 +
k^{2})^{\frac{1-s}{2}} (1 + k^{2})^{\frac{s-1}{2}}
\right ]^{2}.
\end{split}
\end{equation*}
Applying Cauchy Schwartz in $k$, we bound this by
\begin{equation*}
\label{np-key-term}
\begin{split}
\| f \|_{H^{s-1}}^{2} \sum_{n \in \zz}  (1 + n^{2})^{r-1}\sum_{k \in \zz} \frac{|
\wh{g}(n - k) |^{2}}{(1 + k^{2})^{s-1}}.
\end{split}
\end{equation*}
But by change of variables and Fubini
\begin{equation}
\label{opp}
\begin{split}
\sum_{n \in \zz}  (1 + n^{2})^{r-1}\sum_{k \in \zz} \frac{|
\wh{g}(n - k) |^{2}}{(1 + k^{2})^{s-1}}
& = \sum_{k \in \zz}| \wh{g}(k) |^{2} \sum_{n \in \zz}  
\frac{1}{(1 + n^{2})^{s-1}[1 + (n - k)^{2}]^{1-r}}.  
\end{split}
\end{equation}
Without loss of generality, we assume $k \ge 0$ and write 
\begin{equation*}
\begin{split}
&  \sum_{n \in \zz}  
\frac{1}{(1 + n^{2})^{s-1}[1 + (n - k)^{2}]^{1-r}}  
\\
& = 
\sum_{0 \le n \le 2k} \frac{1}{(1 + n^{2})^{s-1}[1 + (n - k)^{2}]^{1-r}} 
+ \sum_{n > 2k} \frac{1}{(1 + n^{2})^{s-1}[1 + (n - k)^{2}]^{1-r}}
\\
& + \sum_{n \ge 0} \frac{1}{(1 + n^{2})^{s-1}[1 + (n + k)^{2}]^{1-r}} 
\\
& \doteq I + II + III.
\end{split} 
\end{equation*}
We have the estimate
\begin{equation}
\label{est-tem}
\begin{split}
II 
& \le \sup_{n > 2k} \frac{1}{\left[ 1 + (n-k)^{2} \right]^{1-r}}
\sum_{n > 2k} \frac{1}{(1 + n^{2})^{s-1}} 
\\
& \lesssim (1 + k^{2})^{r-1}, \quad
s > 3/2.
\end{split}
\end{equation}
Similarly
\begin{equation*}
\begin{split}
III \lesssim (1 + k^{2})^{r-1}, \quad s > 3/2.
\end{split}
\end{equation*}
To estimate $I$, we assume without loss of generality that $k$ is even and write
\begin{equation*}
\begin{split}
&  I = \sum_{0 \le n \le k/2} \frac{1}{(1 + n^{2})^{s-1}[1 + (n - k)^{2}]^{1-r}} 
+ \sum_{k/2 < n \le 3k/2} \frac{1}{(1 + n^{2})^{s-1}[1 + (n - k)^{2}]^{1-r}} 
\\
& + \sum_{3k/2 < n \le 2k} \frac{1}{(1 + n^{2})^{s-1}[1 + (n - k)^{2}]^{1-r}} 
\\
& \doteq i + ii + iii.
\end{split} 
\end{equation*}
Hence, estimating as in \eqref{est-tem}, we have
\begin{equation*}
\begin{split}
i, iii \lesssim (1 + k^{2})^{r-1}, \quad
s > 3/2
\end{split}
\end{equation*}
and
\begin{equation*}
\begin{split}
ii & \le \sup_{k/2 \le n \le 3k/2} \frac{1}{\left( 1 + n^{2} \right)^{s-1}}
\sum_{k/2 \le n \le 3k/2} \frac{1}{[1 + (n-k)^{2}]^{1-r}} \\
& \lesssim \frac{1}{(1 + k^{2})^{s-1}}, \quad r \le 1/2.
\end{split}
\end{equation*}
Therefore, 
\begin{equation*}
\begin{split}
I + II + III & \lesssim (1 + k^{2})^{1-s} + (1 + k^{2})^{r-1}, \quad r \le 1/2, \ s > 3/2
\\
& \lesssim  (1 + k^{2})^{r-1}, \quad r -1 \ge 1-s.
\end{split}
\end{equation*}
Applying this estimate to \eqref{opp} and recalling \eqref{np-key-term},
we obtain
\begin{equation}
\label{yhh}
\begin{split}
\| f g \|_{H^{r-1}} \lesssim \| f \|_{H^{s-1}} \| g \|_{H^{r-1}},
\quad s > 3/2, \ r \le 1/2, \ s + r \ge 2.
\end{split}
\end{equation}
We now need the following result taken from Taylor \cite{Taylor:1995kx}.
%
%
%
%
%
%
%
%
%
\begin{lemma}[Sobolev Interpolation]
For fixed $j \le k, m \le n$ suppose that \\ $T: H^{j} \to H^{m}$ continuously
and $T: H^{k} \to H^{n}$. Then\\ $T: H^{\theta j + (1 - \theta)k} \to H^{\theta
m + (1 - \theta) n}$ continuously for all $\theta \in (0,1]$.
\label{prop:sob-interp}
\end{lemma}
To apply Lemma~\ref{prop:sob-interp}, we note that \eqref{yhh}
and the algebra property of the Sobolev space $H^{t}$, $t > 1/2$ imply that for $s > 3/2$
\begin{equation*}
\begin{split}
\| f g \|_{H^{r-1}} \lesssim \| g \|_{H^{r-1}}, \  \text{where} \ 
r=1/2 \ \text{or} \  r =s, \ \| f \|_{H^{s-1}} =1.
\end{split}
\end{equation*}
That is, for fixed $f \in H^{s-1}$ with $\| f \|_{H^{s-1}} =1$, the map $g \mapsto
Tg = fg$ is linear continuous from $H^{-1/2}$ to $H^{-1/2}$ and from $H^{s-1}$ to
$H^{s-1}$. Therefore, by Lemma~\ref{prop:sob-interp}, it is continuous from
$H^{\theta (s-1) + (1 - \theta)(-1/2) }$ to $H^{\theta (s-1) + (1 - \theta)(-1/2) }$ for all $\theta \in
[0, 1)$. Setting $\theta = (r-1/2)/(s-1/2)$, $ 1/2 \le r < s$, we obtain that $T$ is
continuous from $H^{r-1}$ to $H^{r-1}$. Since $T$ is also linear from $H^{r-1}$
to $H^{r-1}$, we see that 
\begin{equation*}
\begin{split}
\| f g \|_{H^{r-1}} \lesssim \| g \|_{H^{r-1}}, \quad 1/2 \le r \le s, \ s > 3/2, \ \| f \|_{H^{s-1}} =1
\end{split}
\end{equation*}
and so for general $f \in H^{s-1}$ we have 
\begin{equation}
\label{hhh}
\begin{split}
\| f g \|_{H^{r-1}} \lesssim \|f \|_{H^{s-1}}
\| g \|_{H^{r-1}}, \quad 1/2 \le r \le s, \ s > 3/2. 
\end{split}
\end{equation}
Combining \eqref{yhh} and \eqref{hhh} completes the proof in the periodic
case. For the non-periodic case we have
\begin{equation*}
\begin{split}
\| fg\|_{H^{r-1}}^{2}
\le \int_{\rr}  (1 + \xi^{2})^{r-1}\left [ \int_{\rr}
| \wh{f}(\eta) |  | \wh{g}(\xi - \eta) | (1 +
\eta^{2})^{\frac{1-s}{2}} (1 + \eta^{2})^{\frac{s-1}{2}}
d \eta \right ]^{2} d \xi.
\end{split}
\end{equation*}
Applying Cauchy Schwartz in $\eta$, we bound this by
\begin{equation*}
\begin{split}
\| f \|_{H^{s-1}}^{2} \int_{\rr}  (1 + \xi^{2})^{r-1}\int_{\rr} \frac{|
\wh{g}(\xi - \eta) |^{2}}{(1 + \eta^{2})^{s-1}} d \eta d \xi.
\end{split}
\end{equation*}
We now wish to bound the integral term. Applying a change of variable, we see it
is equal to
\begin{equation*}
\begin{split}
\int_{\rr} (1 + \xi^{2})^{r-1} \int_{\rr}
\frac{| \wh{g}(\eta) |^{2}}{[1 + (\xi - \eta)^{2}]^{s-1}} d \eta d \xi
\end{split}
\end{equation*}
which by Fubini is equal to
\begin{equation}
\label{int-pre-calc-lem}
\begin{split}
& \int_{\rr} | \wh{g}(\eta) |^{2} \int_{\rr} \frac{1}{\left[
1 + (\xi - \eta)^{2} \right]^{s-1} (1 + \xi^{2})^{1-r}} d \xi d \eta
\\
& \lesssim \int_{\rr} | \wh{g}(\eta) |^{2} \int_{\rr} \frac{1}{\left[
1 + |\xi - \eta| \right]^{2(s-1)} (1 + |\xi|)^{2(1-r)}} d \xi d \eta.
\end{split}
\end{equation}
We now need the following lemma. 
\begin{lemma}
\label{lem:calc}
Fix $p, q > 0$ such that $p +q >1$, and let $r =\min\left\{p - \ee_{q}, q -
\ee_{p}, p+q-1 \right\}$, where $\ee_{j} > 0$ is arbitrarily small for $j = 1$
and $\ee_{j} = 0$ for $j \neq 1$. Adopt the notation
$\langle x - \alpha \rangle  \doteq 1 + | x - \alpha |$. Then 
\begin{equation*}
\begin{split}
& \int_{\rr} \frac{1}{\langle x - \alpha \rangle ^{p} \langle x -
\beta \rangle
^{q}} d x
\le \frac{c_{p,q, \ee}}{\langle \alpha - \beta \rangle ^{r}}. 
\end{split}
\end{equation*}
\end{lemma}
To be able to apply the lemma to the integral term in \eqref{int-pre-calc-lem}, 
we must first check
that its conditions are met. Let $ s = 3/2 + \ee$, $r = 1- \delta$, $\ee > 0$, $
\delta \ge 0$ and observe that
\begin{equation*}
\begin{split}
2(s-1) + 2(1-r)
& = 2(s-r)
\\
& = 2[3/2 + \ee - (1 - \delta)]
\\
& = 2(1/2 + \ee + \delta)
\\
& = 1 + 2 \ee + 2 \delta > 1.
\end{split}
\end{equation*}
Furthermore, $2(s-1), 2(1-r) > 0$. Hence, Lemma~\ref{lem:calc} is applicable. 
Note that since $s > 3/2$, we see that $2(s-1) \neq 1$. However, it is possible that $2(1-r) =1$; hence we must now separate the cases $r \neq 1/2$ and $r = 1/2$. Suppose $r \neq 1/2$. Then 
\begin{equation*}
\begin{split}
\min\left\{ 2(s-1), 2(1-r), 2(s-1) + 2(1-r) -1 \right\}
& = \min\left\{ 1 + 2 \ee, 2 \delta, 2\ee + 2 \delta \right\}
\\
& = \min\left\{ 1 + 2 \ee, 2 \delta\right\}
\\
& = 2 \delta, \quad \delta \le 1/2 + \ee.
\end{split}
\end{equation*}
If $r = 1/2$, then since $s > 3/2$, we can choose $\eta > 0$ sufficiently small
such that
\begin{equation*}
\begin{split}
\min\left\{ 2(s-1) -\eta , 2(1-r), 2(1-r) + 2(s-1) - 1  \right\}
& = 1 
\\
& = 2(1 -r)
\\
& = 2\delta.
\end{split}
\end{equation*}
Hence, for $0 \le \delta \le 1/2 + \ee$, $\ee >
0$, \eqref{int-pre-calc-lem} is bounded by
\begin{equation*}
\begin{split}
C_{s,r} \int_{\rr}  | \wh{g}(\eta) |^{2} \int_{\rr} \frac{1}{\left( 1
+ |\eta| \right )^{2 \delta}} d \xi d \eta 
& \le
\| g \|_{H^{-\delta}}^{2}
\\
& = \| g \|_{H^{r-1}}^{2}.
\end{split}
\end{equation*}
Our restriction on $\delta$ is equivalent to the restriction 
$$1-r \le 1/2 + s - 3/2, \quad r \le 1, \ s > 3/2,$$ or
$$s + r \ge 2,  \quad  r \le 1, \ s > 3/2.$$ Therefore, 
\begin{equation*}
\begin{split}
\| f g \|_{H^{r-1}} \lesssim \| f \|_{H^{s-1}} \| g \|_{H^{r-1}},
\quad s + r \ge 2, \ s > 3/2, \ r \le 1.
\end{split}
\end{equation*}
The remainder of the proof is analogous to that in the periodic case.
\end{proof}
\begin{proof}[Proof of Lemma~\ref{lem:calc}]
By the change of variable $x \mapsto x/2 + (\alpha + \beta)/2$, we have
\begin{equation}
\label{rur}
\begin{split}
\int_{\rr} \frac{1}{\langle x - \alpha \rangle^{p} \langle  x -
\beta
\rangle^{q}}d x
& \simeq \int_{\rr} \frac{1}{\langle x/2 - (\alpha - \beta)/2  \rangle^{p}
\langle  x/2 + (\alpha - \beta)/2 \rangle^{q}} d x
\\
& \lesssim \int_{\rr} \frac{1}{\langle x - (\alpha - \beta)  \rangle^{p}
\langle  x + (\alpha - \beta) \rangle^{q}} d x
\\
& = \int_{\rr} \frac{1}{\langle a - x \rangle ^{p} \langle a + x \rangle
^{q}} d x, \quad a = \alpha - \beta
\end{split}
\end{equation}
which for $a =0$ reduces to 
\begin{equation*}
\begin{split}
\int_{\rr} \frac{1}{\langle x \rangle ^{p+q}} d x 
& = 2 \int_{0}^{\infty} \frac{1}{(1 + x)^{p+q}} d x
\\
& = \frac{2}{p+q -1}.
\end{split}
\end{equation*}
We now handle the case $a \neq 0$. Note that by the change of variable $x \mapsto
-x$ we may restrict our attention to the case  $a > 0$ without loss of
generality. Split
\begin{equation*}
\begin{split}
\int_{\rr} \frac{1}{\langle a + x \rangle ^{p} \langle a - x \rangle
^{q}} d x
& = \int_{-2a}^{2a}
\frac{1}{\langle a + x \rangle ^{p} \langle a - x \rangle
^{q}} d x
\\
& + \int_{| x | \ge 2a} 
\frac{1}{\langle a + x \rangle ^{p} \langle a - x \rangle
^{q}} d x
\\
& = I + II.
\end{split}
\end{equation*}
Then
\begin{equation*}
\begin{split}
I 
& = \int_{0}^{2a}
\frac{1}{\langle a + x \rangle ^{p} \langle a - x \rangle
^{q}} d x + \int_{-2a}^{0}
\frac{1}{\langle a + x \rangle ^{p} \langle a - x \rangle
^{q}} d x.
\end{split}
\end{equation*}
We bound the first term by
\begin{equation*}
\begin{split}
\sup_{0 \le x \le 2a} \frac{1}{\langle a + x \rangle
^{p}} \int_{0}^{2a} \frac{1}{\langle a - x \rangle ^{q}} d x
& = \frac{1}{\langle a \rangle ^{p}} \int_{0}^{2a} \frac{1}{(1 + | a -
x
|)^{q}} d x  
\\
& = \frac{2}{\langle a \rangle ^{p}} \int_{0}^{a} \frac{1}{(1 + a -
x)^{q}} d x
\\
& \lesssim
\begin{cases}
1/{\langle a \rangle ^{p}} \left| 1 - 1/{(1 +
a)^{q -1}} \right|, \quad & q \neq 1
\\
\log(1+a)/{\langle a \rangle^{p} }, \quad & q =1.
\end{cases}
\end{split}
\end{equation*}
But
\begin{equation*}
\begin{split}
\frac{1}{\langle a \rangle ^{p}}\left| 1 - \frac{1}{(1 +
a)^{q -1}} \right|
& \lesssim
\begin{cases}
1/{\langle a \rangle^{p} }, \quad & q > 1
\\
1/{\langle a \rangle ^{p + q -1}}, \quad & q < 1
\end{cases}
\end{split}
\end{equation*}
and
\begin{equation*}
\begin{split}
\frac{\log(1 + a)}{\langle a \rangle^{p} } \le  \frac{c_{\ee}}{\langle a
\rangle ^{p - \ee}} \ \text{for any} \ \ee > 0.
\end{split}
\end{equation*}
For the second term, we bound by
\begin{equation*}
\begin{split}
& \sup_{-2a \le x \le 0} \frac{1}{\langle a - x \rangle
^{q}} \int_{-2a}^{0} \frac{1}{\langle a + x \rangle ^{p}} d x
\\
& = \frac{1}{\langle a \rangle ^{q}} \int_{-2a}^{0} \frac{1}{(1 + | a +
x
|)^{p}} d x 
\\
& = \frac{2}{\langle a \rangle ^{q}} \int_{-a}^{0} \frac{1}{(1 + a +
x)^{p}} d x
\\
& \lesssim
\begin{cases}
1/{\langle a \rangle ^{q}} \left| 1 - 1/{(1 +
a)^{p -1}} \right|, \quad & p \neq 1
\\
\log(1+a)/{\langle a \rangle^{q} }, \quad & p =1.
\end{cases}
\end{split}
\end{equation*}
But
\begin{equation*}
\begin{split}
\frac{1}{\langle a \rangle ^{q}}\left| 1 - \frac{1}{(1 +
a)^{p -1}} \right|
& \lesssim
\begin{cases}
1/{\langle a \rangle ^{q}}, \quad & p > 1
\\
1/{\langle a \rangle ^{p + q -1}}, \quad & p < 1
\end{cases}
\end{split}
\end{equation*}
and
\begin{equation*}
\begin{split}
\frac{\log(1 + a)}{\langle a \rangle^{q} } \le  \frac{c_{\ee}}{\langle a
\rangle ^{q - \ee}} \ \text{for any} \ \ee > 0.
\end{split}
\end{equation*}
Therefore,
\begin{equation*}
I \le  \frac{c_{p,q, \ee}}{\langle a \rangle ^{\min\left\{ p-\ee_{q}, q -\ee_{p}, p + q-1 \right\}}}.
\end{equation*}
Also
\begin{equation*}
\begin{split}
II 
& = \int_{x \ge 2a} \frac{1}{(1 + x - a)^{p} (1 + x +
a)^{q}} d x
\\
& \le \int_{x \ge 2a} \frac{1}{(1 + x -a)^{p+q}} d x
\\
& \simeq \frac{1}{\langle a \rangle^{p+q -1}}, \qquad p + q > 1.
\end{split}
\end{equation*}
Collecting our estimates for $I$ and $II$ we see that for 
$p, q > 0$ such that $p +q >1$, and $r =\min\left\{p -\ee_{q}, q - \ee_{p}, p+q-1
\right\}$, we have
\begin{align*}
\int_{\rr} \frac{1}{\langle a - x \rangle ^{p} \langle a + x \rangle
^{q}} d x
\le \frac{c_{p,q, \ee}}{\langle a \rangle ^{r}}.
\label{est-2}
\end{align*}
Recalling \eqref{rur}, the proof is complete.
\end{proof}
\subsection*{Acknowledgements} The author
thanks the Department of Mathematics at the University of Notre Dame and
the Arthur J. Schmitt Foundation for supporting his doctoral studies.
\newcommand{\etalchar}[1]{$^{#1}$}
\providecommand{\bysame}{\leavevmode\hbox to3em{\hrulefill}\thinspace}
\providecommand{\MR}{\relax\ifhmode\unskip\space\fi MR }
\providecommand{\MRhref}[2]{%
\href{http://www.ams.org/mathscinet-getitem?mr=#1}{#2}
}
\providecommand{\href}[2]{#2}

%
%
%
%
\end{document}